\documentclass[11pt]{amsart}

\usepackage[utf8]{inputenc} 

\usepackage[width=16cm,height=24cm]{geometry} 

\usepackage[backend=biber,
style=numeric,
url=false, doi=true, isbn=false,
, maxnames=10 
, sorting=nyt 
,giveninits=true 
]{biblatex}
\usepackage[autostyle=true]{csquotes} 

\usepackage{amsfonts,amsmath,amssymb,amsthm}
\usepackage{graphics,graphicx,epsfig}

\usepackage{comment}
\usepackage{color}
\usepackage{bbm}
\usepackage{dsfont} 
\usepackage{enumitem}
\usepackage{appendix}
\usepackage{tikz, tikz-cd}

\newcommand{\lpv}{{L^{p(\cdot)}(\Omega)}}

\newtheorem{thm}{Theorem}[section]

\newtheorem{prop}[thm]{Proposition}
\newtheorem{exa}[thm]{Examples}

\begin{document}
\setlength{\baselineskip}{6mm}
\title[A note on the strict sigularity of the inclusion between Nakano sequence spaces]{A note on the strict sigularity of the inclusion between Nakano sequence spaces}

\author[M. Sanchiz]{Mauro Sanchiz}

\address{
Departamento de An{\'a}lisis Matem{\'a}tico y Matem\'atica Aplicada, Facultad de Matem{\'a}ticas, Universidad Complutense, 28040 Madrid, Spain}
  \email{
  (M. Sanchiz) msanchiz@ucm.es}

\thanks{Supported by NAWA under the Ulam postoctoral program 2024/1/00064 and and partially supported by the research project 2025/00145/001 ``Operadores, retículos y estructura de espacios de Banach'', grant PID2019-107701G-I00 and scholarship CT42/18-CT43/18}


\begin{abstract}
We characterize the strictly singular inclusions $\ell_{p_n}\hookrightarrow\ell_{q_n}$ between Nakano sequence spaces providing a useful criterion, namely $\varliminf_{n\rightarrow\infty}\vert p_n-q_n\vert>0$ (also recently obtained by Lang and Nekvinda in \cite{lang_embeddings_2025} with a different proof). It is also noted that no inclusion operator between Nakano sequence spaces is compact, neither $L$-weakly compact nor $M$-weakly compact. An easy criterion is given for the weak compactness of the inclusion.

\end{abstract}

\maketitle

\section{Introduction and Preliminaries} 

An operator $T : X \rightarrow Y$ between two Banach spaces is said to be \textit{strictly singular} (or Kato) if there is no infinite-dimensional (closed) subspace $Z$ of $X$ such that the restricted operator $T\vert_{Z}$ is an isomorphism (cf. \cite{kato_perturbation_1958}). 
The study of strictly singular inclusions has been quite extensive for symmetric function spaces (cf. \cite{astashkin_strictly_2009,hernandez_strictly_2003, hernandez_strict_2003}, Section 5 in \cite{kalton_orlicz_1977}). 

On this note, we study the strict singularity of the inclusion operator 
$$i:\ell_{p_n}\hookrightarrow\ell_{q_n}$$
between the Banach spaces of the non-symmetric class of Nakano sequence spaces. We also show some examples and briefly note the compactness, $L$-weakly compactness, $M$-weakly compactness and weakly compactness properties of the inclusion.

The theorems on strict singularity in this article can be deduced from the more general results (including the quasi-Banach case) obtained by Lang and Nekvinda in \cite{lang_embeddings_2025}. Without novelty in strict singularity results, this article has interest in itself due to its proofs, which differ from those of \cite{lang_embeddings_2025}. Lang and Nekvinda prove their statements by directly computing series, while our approach is different and is based on basic sequences\footnote{The results in this article were independently proven by the author on his thesis \cite{sanchiz_tesis_2023}, defended on July 2023, but submitted on February 2023, before the first version of \cite{lang_embeddings_2025} submitted to ArXiv on April 2023 \cite{lang_embeddings_2023}. The reason this article was not written earlier is that the results were a particular case of the general study of disjointly strictly singular inclusions between Nakano spaces $\lpv$, recently finished and published in \cite{hernandez_disjointly_2025}. In the final version of that article, the sequence Nakano spaces were removed since they did not fit well.}.

Given a real sequence $(p_n)$ with $1\leq p_{n}< \infty$ (named {\it exponent sequence}), the {\em  Nakano sequence space} $\ell_{p_n}$ is the space of all real sequences $(x_n)\in \mathbb R^{\mathbb N}$ such that the modular $\rho_{(p_n)}\left(\left(\frac{x_n}{r}\right)\right)<\infty$ for some $r>0$, where
$$
\rho_{(p_n)}\left(\left(x_n\right)\right):=\sum_{n=1}^\infty\left|x_n\right|^{p_n}.
$$
The space $\ell_{p_n}$ is a Banach space equipped with the Luxemburg norm:
$$
\lVert (x_n)\rVert_{\ell_{p_n}}=\inf \left\{ r>0: \rho_{(p_n)}\left(\left(\frac{x_n}{r}\right)\right)=\sum_{n=1}^{\infty} \left\vert \frac{x_n}{r}\right\vert^{p_n}\leq 1 \right\}.
$$

The Nakano sequence spaces are a particular class of Banach function spaces 
and Musielak-Orlicz sequence spaces (or modular sequence spaces) (cf. \cite{bennett_interpolation_1988, diening_lebesgue_2011,lindenstrauss_classical_1996}).

To study isomorphisms between Nakano sequence spaces, we can just study set inclusions due to the following well known fact in Banach function spaces: if a Banach function space $E(\mu)$ is contained (as sets) in another Banach function space $F(\mu)$ over the same measure space $(\Omega,\mu)$, then the inclusion operator $E(\mu)\hookrightarrow F(\mu)$ is well defined and bounded. Hence, if $E(\mu)=F(\mu)$ as sets, they are isomorphic (cf. \cite{bennett_interpolation_1988} Thm.1.8, Cor.1.9.).

Two Nakano sequence spaces $\ell_{p_n}$ and $\ell_{q_n}$ are isomorphic if the exponent sequences $(p_n)$ and $(q_n)$ are close enough (in particular, if $p_n=q_n$ up to finite many terms):

\begin{prop}[Nakano's Lemma, \cite{nakano_modulared_1951}]\label{Nakano-igualdad}
Let $1\leq p_{n},\,q_{n} <\infty$.  Then \,$\ell_{p_{n}} =  \ell_{q_{n}}$ if and only if there exists $\alpha > 0$ such that  
$$
\sum_{n=1}^{\infty} \alpha^{\frac{p_{n}q_{n}}{|p_{n}-q_{n}|}}  < \infty.
$$
\end{prop}

Given a subset $A=\{n_k\}_{k\in I}\subset \mathbb{N}$, we denote 
$$
\ell_{p_n}(A):=\left\{(x_k)\in\mathbb R^{I}:\rho_{(p_n)}\left(\left(\frac{x_k}{r}\right)\right)=\sum_{k\in I}^\infty\left|\frac{x_k}{r}\right|^{p_{n_k}}\,<\infty,\mbox{ for some }r>0\right\}.
$$
For a partition $A\sqcup B=\mathbb{N}$, we have $\ell_{p_n}=\ell_{p_n}(A)\oplus\ell_{p_n}(B)$. We can also consider exponent sequences $(p_n)$ with $1\leq p_n\leq \infty$ by denoting $A=\{n\in\mathbb{N}: p_n=\infty\}$ and taking $\ell_{p_n}=\ell_\infty(A)\oplus\ell_{p_n}(B)$.
Denote the upper limit of a sequence $\varlimsup_n p_n=\lim_n\sup_{k\geq n} p_k$ and the lower limit $\varliminf_n p_n=\lim_n\inf_{k\geq n} p_k$. A Nakano sequence space $\ell_{p_n}$ is separable if $\varlimsup_n p_n<\infty$ and reflexive if, moreover, $\varliminf_n p_n>1$. This occurs since: 

\begin{prop}\label{ele.infinito}
    Let $1\leq p_n<\infty$ be an unbounded exponent sequence. Then there exists a subsequence $(p_{n_k})$ such that $\ell_{p_{n_k}}=\ell_\infty$.  
\end{prop}

\begin{proof}
    Consider a subsequence $(p_{n_k})$ satisfying $p_{n_k}\geq k$. Let us show that $\ell_{p_{n_k}}=\ell_\infty$.
    If $(x_k)\not\in\ell_\infty$, then for every $\lambda>0$,
    $$
    \sum_{k=1}^{\infty} \left( \frac{\vert x_k\vert}{\lambda}\right)^{p_{n_k}}=\infty,
    $$
    so $(x_n)\not\in\ell_{p_{n_k}}$.
    If $(x_k)\in\ell_\infty$, then for $\lambda=2\lVert (x_n)\rVert_{\infty}$,
    $$
    \sum_{k=1}^{\infty} \left( \frac{\vert x_k\vert}{2\lVert (x_n)\rVert_{\infty}}\right)^{p_{n_k}}
    \leq \sum_{k=1}^{\infty} \left( \frac{1}{2}\right)^{k}=1<\infty,
    $$
    so $(x_k)\in \ell_{p_{n_k}}$.
\end{proof}

Also, for separable Nakano sequence spaces, a sequence $(x_n^k)_k$ is seminormalized if and only if its modular is lower and upper bounded.

Just like in classical sequence spaces $\ell_p-\ell_q$, the inclusion between Nakano sequence spaces $i: \ell_{p_n}\hookrightarrow\ell_{q_n}$ holds from below to above, i.e. $p_n\leq q_n$ (see \cite{diening_lebesgue_2011} Lemma 3.3.6). Nonetheless it is possible that $\ell_{p_n}\subset\ell_{q_n}$ even if $p_n\geq q_n$ as the following result shows:

\begin{thm}[cf. \cite{diening_lebesgue_2011} Theorem 3.3.7]
Let $1\leq p_n, q_n\leq \infty$ for every $n$ and $\frac{1}{r_n}:=\max\{0,\frac{1}{q_n}-\frac{1}{p_n}\}$. If $\mathds{1} \in\ell_{r_n}$, then the inclusion operator $i:\ell_{p_n}\hookrightarrow\ell_{q_n}$ holds.
\end{thm}

Note that, pairing with Proposition \ref{Nakano-igualdad}, this result means that the inclusion $i:\ell_{p_n}\rightarrow\ell_{q_n}$ can hold whenever $p_n\geq q_n$ if $(p_n)$ and $(q_n)$ are close enough. But so close that even $\ell_{p_n}=\ell_{q_n}$.

The canonical unit sequence $(e_n)$ is a disjoint Schauder basis in every separable $\ell_{p_n}$ (see \cite{lindenstrauss_classical_1996} Thm.4.d.3). We refer to \cite{lindenstrauss_classical_1996} for basic notions and results of Schauder basis.
Instead of proving that the inclusion $i:\ell_{p_n}\hookrightarrow\ell_{q_n}$ is strictly singular, we will just show that the inclusion is \textit{disjointly strictly singular} (meaning that the inclusion $i$ is not an isomorphism for any closed subspace spanned by a normalized pairwise disjoint sequence $(x_{n}^k)_k$ in $\ell_{p_n}$), due to the following:

\begin{prop}[\cite{hernandez-hernandez_disjointly_1990} Prop.1]\label{DSS=SS}
Let $E$ be a Banach lattice with a disjoint Schauder basis and $Y$ be a Banach space. An operator $T: E\rightarrow Y$ is disjointly strictly singular if and only if it is strictly singular.
\end{prop}

\section{Main Results}

We consider first the inclusions between separable Nakano sequence spaces.

\begin{thm}\label{Thm.base}
Let $1\leq p_n, q_n\leq M<\infty$ and the inclusion $i:\ell_{p_n}\hookrightarrow\ell_{q_n}$ hold. Then, the inclusion $i$ is strictly singular if and only if
$$
\varliminf_{n\rightarrow\infty}\vert p_n-q_n\vert>0
$$
i.e., the exponent sequences $(p_n)$ and $(q_n)$ do not share any common accumulation point.
\end{thm}

\begin{proof}
$(\Rightarrow):$ Suppose that there exists a subsequence $(n_k)$ such that $\vert p_{n_k}-q_{n_k}\vert \xrightarrow{k\rightarrow\infty}0$. Then, up to passing to another subsequence, we can suppose that $\vert p_{n_k}-q_{n_k}\vert<\frac{1}{k}$. Taking $\alpha=\frac{1}{2}$, we have that
    $$
    \sum_{k=1}^{\infty}\alpha^{\frac{p_{n_k}q_{n_k}}{\vert p_{n_k}-q_{n_k}\vert}}
    \leq \sum_{k=1}^{\infty}\left(\frac{1}{2}\right)^{\frac{1}{\vert\frac{1}{k}\vert}}
    =\sum_{k=1}^{\infty} \frac{1}{2^k}
    =1
    <\infty,
    $$
    and thus, by Proposition \ref{Nakano-igualdad}, the sequence spaces 
    $$
    [e_{n_k}]_{\ell_{p_n}}\simeq\ell_{p_{n_k}}=\ell_{q_{n_k}}\simeq[e_{n_k}]_{\ell_{q_n}}
    $$
    coincide. Therefore, the inclusion $i: \ell_{p_n}\hookrightarrow \ell_{q_n}$ is not strictly singular.
    
    $(\Leftarrow):$ Let $\varliminf_{n\rightarrow\infty} \vert p_n-q_n\vert=\varepsilon>0$. Let $(s_j)$ be a pairwise disjoint normalized sequence in $\ell_{p_n}$. Then, $(s_j)$ is a basic sequence, so there exists an $1\leq r<\infty$ such that $\ell_r$ is isomorphic to a complemented subspace in $\ell_{p_n}$ (\cite{woo_modular_1973} Thm.3.8) generated by some block basic subsequence of $(s_j)$ (besides, of disjoint blocks) (\cite{lindenstrauss_classical_1996} Prop.1.a.11). Further, if $i\vert_{[s_j]}$ was an isomorphism, then $i$ would also be an $\ell_r$-isomorphism when restricted to the block basic subsequence of $(s_j)$. So, we can suppose without loss of generality that $(s_j)$ is itself a block basic subsequence of the canonical basis $(e_n)$ in $\ell_{p_n}$ which is equivalent to the canonical basis of $\ell_r$.
    
    Let us denote
    $$
    s_j=\sum_{i\in\sigma_j}a_{ji} \, e_i \,,
    $$
    where every $\sigma_j\subset\mathbb{N}$ is finite and $\max \sigma_{j}<\min \sigma_{j+1}$ for every $j$.
    Now we can write
    $$
    s_j=v_j+u_j \,,
    $$
    where
    $$
    v_j=\sum_{\{i\in\sigma_j : \, p_i\leq r-\frac{\varepsilon}{2}\}} a_{ji}\, e_i
    $$
    and 
    $$
    u_j=\sum_{\{i\in\sigma_j :\, p_i>r-\frac{\varepsilon}{2}\}} a_{ji}\, e_i.$$    
    We distinguish now two cases:
    \begin{itemize}
        \item [$(1)$] If $\lVert v_j\rVert_{p_n}\rightarrow 0$ (or up to a subsequence), we have $[s_j]_{\ell_{p_n}}\simeq [u_j]_{\ell_{p_n}}$ (again up to subsequence, cf. \cite{lindenstrauss_classical_1996} Prop.1.a.9). Then, since $q_i> p_i + \varepsilon> r-\frac{\varepsilon}{2}+\varepsilon = r+ \frac{\varepsilon}{2}$ for $i\in\sigma_j$ up to finite many $i$ (note that if $p_i>q_i$ for infinite many $i$ while $\varliminf_{n\rightarrow\infty} \vert p_n-q_n\vert=\varepsilon>0$, the inclusion does not hold), we have $[i(u_j)]_{\ell_{q_n}}\subset \ell_{q_n}(A)$ for $A=\{n\in\mathbb{N}: q_n>r+\frac{\varepsilon}{2}\}$, and this subspace cannot be isomorphic to $\ell_r$ because $r$ is not an accumulation point of $\{q_n: q_n>r+\frac{\varepsilon}{2}\}$ (\cite{peirats_lp-copies_1992} Theorem 2.9), which gives a contradiction.
        
        \item [$(2)$] Assume now $1\geq \lVert v_j\rVert_{p_n}\geq\delta>0$ (up to some finite many $j$). Consider the subspace $[v_j]_{\ell_{p_n}}$ generated by $(v_j)$ in $\ell_{p_n}(B)$ for $B=\{n\in\mathbb{N}: p_n\leq r-\frac{\varepsilon}{2}\}$. Reasoning as above, there exists an scalar $k\leq r-\frac{\varepsilon}{2}<r$ such that the canonical basis of $\ell_k$ is equivalent to certain block basic subsequence $(t_m)$ of $(v_j)$. Let us write
        $$
        t_m=\sum_{j\in H_m} b_{mj} v_j,
        $$
        with $H_m\subset\mathbb{N}$ finite for every $m$, and $\max H_m<\min H_{m+1}$. Let us define
        $$
        \tilde{t}_m:=\sum_{j\in H_m}b_{mj} s_j
        =\sum_{j\in H_m}b_{mj} v_j + \sum_{j\in H_m}b_{mj} u_j
        =t_m + \sum_{j\in H_m}b_{mj} u_j
        $$
        and let us see that $(\tilde{t}_m)$ is a seminormalized block basic subsequence of $(s_j)$. If $\rho_{(p_n)}(u_j)\xrightarrow{j\rightarrow\infty}0$, then, up to a subsequence, we have $[v_j]_{\ell_{p_n}}\simeq[s_j]_{\ell_{p_n}}$ which leads to a contradiction as $[v_j]_{\ell_{p_n}}\simeq[s_j]_{\ell_{p_n}}\simeq\ell_r$ by hypothesis, so it cannot have a subspace ($[t_m]_{\ell_{p_n}} \subset [v_j]_{\ell_{p_n}}$) isomorphic to $\ell_k$ for $k< r$. Therefore,
        $$
        1\geq\rho_{(p_n)}(v_j)\geq\delta_1>0
        $$
        and
        $$
        1\geq\rho_{(p_n)}(u_j)\geq\delta_2>0.
        $$
        
        Also, $t_m$ is seminormalized, so there exist $c,C>0$ with $c\leq \rho_{(p_n)}(t_m)\leq C$ for every $m$. Hence, we have
        \begin{align*}
            C
            \geq{} &   \rho_{(p_n)}\left(\sum_{j\in H_m} \vert b_{mj}\vert v_j \right)
            = \sum_{j\in H_m}\left(\sum_{\{i\in\sigma_j, p_i\leq r-\frac{\varepsilon}{2}\}} \vert b_{mj}\vert^{p_i}\vert a_{ji}\vert^{p_i}\right)\\
            \geq{}    &   \sum_{j\in H_m}\vert b_{mj}\vert ^{r-\frac{\varepsilon}{2}}\rho_{(p_n)}(v_j)\frac{\delta_1}{\delta_1}
            \geq  \delta_1 \sum_{j\in H_m} \vert b_{mj}\vert^{r-\frac{\varepsilon}{2}}\rho_{(p_n)}(u_j)\\
            ={}    &   \delta_1 \sum_{j\in H_m}\vert b_{mj}\vert^{r-\frac{\varepsilon}{2}}\left(\sum_{\{i\in\sigma_j, p_i>r-\frac{\varepsilon}{2}\}} \vert a_{ji}\vert^{p_i}\right)\\
            \geq{}  &   \delta_1 \sum_{j\in H_m} \left(\sum_{\{i\in\sigma_j, p_i>r-\frac{\varepsilon}{2}\}}\vert b_{mj}\vert^{p_i}\vert a_{ji}\vert^{p_i}\right)\\
            ={} &   \delta_1 \, \rho_{(p_n)}\left(\sum_{j\in H_m}b_{mj} u_j \right).
        \end{align*}
        Thus, we have $\frac{C}{\delta_1}\geq \rho_{(p_n)}\left(\sum_{j\in H_m} b_{mj}u_j\right)$ and, therefore,
        $$
        c
        \leq\rho_{(p_n)}(t_m)
        \leq \rho_{(p_n)}(\tilde{t}_m)
        \leq \rho_{(p_n)}(t_m) + \rho_{(p_n)}\left( \sum_{j\in H_m} b_{mj}u_j\right)
        \leq 1+\frac{C}{\delta_1}.
        $$
        Hence $(\tilde{t}_m)$ is a semi-normalized block basic subsequence of $(s_j)$.
    \end{itemize}
        
    Finally, let us see that $[\tilde{t}_m]_{\ell_{p_n}} = \ell_k$, which will give a contradiction since $[\tilde{t}_m]_{\ell_{p_n}}\subset [s_j]_{\ell_{p_n}}= \ell_r$. Assume that
    $$
    \rho_{(p_n)}\left(\sum_{m=1}^{\infty} x_m \tilde{t}_m\right)
    <\infty.
    $$
    Then,
    $$
    \rho_{(p_n)}\left(\sum_{m=1}^{\infty} x_m t_m\right)
    <\infty
    $$
    and it follows that $(x_m)\in\ell_k$ since $[t_m]_{\ell_{p_n}}=\ell_k$. Thus, $[\tilde{t}_m]_{\ell_{p_n}}\subset \ell_k$.
    
    On the other hand, if $(x_m)\in\ell_k \subset\ell_r$, then $\rho_{(p_n)}\left(\sum_m x_m \left( \sum_{j\in H_m} b_{mj} v_j \right) \right)<\infty$. Furthermore, the basic sequence $(s_j)$ is equivalent to the canonical basis of $\ell_r$ and $(\tilde{t}_m)$ is a seminormalized block basis of $(s_j)$, so $(\tilde{t}_m)$ is equivalent to the canonical basis of $\ell_r$ too. Hence, if $(x_m)\in\ell_k\subset\ell_r$, we have that
    $$
    \rho_{(p_n)}\left(\sum_{m=1}^{\infty} x_m\tilde{t}_m\right)
    <\infty.
    $$
    Thus, $\ell_k\subset [\tilde{t}_m]_{\ell_{p_n}}$ and we arrive to
    $$
    \ell_k\simeq[\tilde{t}_m]_{\ell_{p_n}}\subset [s_j]_{\ell_{p_n}}\simeq\ell_r,
    $$
    what gives a contradiction.
\end{proof}

We can extend Theorem \ref{Thm.base} easily to the case of unbounded exponent sequences $(q_n)$:

\begin{thm}\label{Thm.infinito1}
    Let $1\leq p_n\leq M<\infty$ and $1\leq q_n\leq\infty$ with $\varlimsup_n q_n=\infty$ and let the inclusion $i:\ell_{p_n}\hookrightarrow\ell_{q_n}$ hold. Then the inclusion $i$ is strictly singular if and only if
    $$
    \varliminf_{n\rightarrow\infty}\vert p_n-q_n\vert>0.
    $$
\end{thm}

\begin{proof}
    Take $A=\{n\in\mathbb{N}: q_n\leq M+1\}$ and $B=\mathbb{N}\setminus A$. Then $\ell_{p_n}=\ell_{p_n}(A)\oplus\ell_{p_n}(B)$ and $\ell_{q_n}=\ell_{q_n}(A)\oplus\ell_{q_n}(B)$ and the inclusion $i$ is strictly singular if both inclusions $\ell_{p_n}(A)\hookrightarrow\ell_{q_n}(A)$ and $\ell_{p_n}(B)\hookrightarrow \ell_{q_n}(B)$ are strictly singular. For $\ell_{p_n}(A)\hookrightarrow\ell_{q_n}(A)$ we apply Theorem \ref{Thm.base}, since $(q_n)_{\vert A}$ is bounded, getting that the inclusion is strictly singular if and only if
    $$
    \varliminf_{n\rightarrow\infty,\, n\in A}\vert p_n-q_n\vert>0.
    $$
    For $\ell_{p_n}(B)\hookrightarrow \ell_{q_n}(B)$, since $p_n\leq M< M+1\leq q_n$ for every $n\in B$, we can factorize $\ell_{p_n}(B)\hookrightarrow \ell_M(B)\hookrightarrow \ell_{M+1}(B) \hookrightarrow \ell_{q_n}(B)$, where $\ell_M\hookrightarrow \ell_{M+1}$ is strictly singular, hence the inclusion is strictly singular. On the other hand, it holds that
    $$
    \varliminf_{n\rightarrow\infty,\, n\in B}\vert p_n-q_n\vert \geq 1>0.
    $$
    In conclusion, $i$ is strictly singular if and only if
    $$
    \varliminf_{n\rightarrow\infty}\vert p_n-q_n\vert >0.
    $$
\end{proof}

In the case of the exponent sequence $(p_n)$ be unbounded, then the exponent sequence $(q_n)$ is also unbounded (or either the inclusion does not hold) and the inclusion is not strictly singular since both Nakano sequence spaces $\ell_{p_n}$ and $\ell_{q_n}$ share an $\ell_{\infty}$ copy.

\begin{thm}\label{Thm.infinito2}
    Let $1\leq p_n,q_n\leq\infty$ with $\varlimsup_n p_n=\infty$ and let the inclusion $i:\ell_{p_n}\hookrightarrow\ell_{q_n}$ hold. Then the inclusion $i$ is not strictly singular.
\end{thm}

\begin{proof}
    Since $i:\ell_{p_n}\hookrightarrow\ell_{q_n}$ holds, the inclusion $\ell_{p_{n_k}}\hookrightarrow\ell_{q_{n_k}}$ also holds for every subsequence $({n_k})$. In particular, take a subsequence $(n_k)$ with $n_k>k$, where $\ell_{p_{n_k}}=\ell_\infty$ by Proposition \ref{ele.infinito}. Thus, $\ell_\infty\hookrightarrow \ell_{q_{n_k}}$. So, $\ell_{q_{n_k}}=\ell_\infty$ and $\ell_{p_{n_k}}\hookrightarrow\ell_{q_{n_k}}$ is an isomorphism, hence $i$ is not strictly singular.
\end{proof}

We give now a few examples to illustrate the different situations we can have for the inclusion $i:\ell_{p_n}\hookrightarrow\ell_{q_n}$. First, we give a trivial example, and then we construct approaching but separated enough sequences $(p_n)$ and $(q_n)$ to apply the results.
\begin{exa}
{\hfill}

\begin{enumerate}
    \item Let $p_n=1+\frac{1}{n}$ and $q_n=n$. The inclusions $\ell_1\hookrightarrow\ell_{p_n}$ and $\ell_{q_n}\hookrightarrow\ell_{\infty}$ are not strictly singular since even $\ell_{p_n}=\ell_1$ and $\ell_{q_n}=\ell_{\infty}$ by Propositions \ref{Nakano-igualdad} and \ref{ele.infinito}. The inclusion $i:\ell_{p_n}\hookrightarrow\ell_{q_n}$ is strictly singular by Theorem \ref{Thm.infinito1}.
\end{enumerate}
Now, let us take the sequence $(a_n)$ defined by $a_n:=\min\{k\in\mathbb{N}: n\leq \sum_{j=1}^k j^j\}$, where each natural $j$ appears $j^j$ times\footnote{$(a_n)=(1,2,2,2,2,3,...)$ followed by $3^3=9$ terms of $3$'s, then $4^4=256$ terms of $4$'s, then $5^5$ terms of $5$'s...}. Then, given $\alpha>0$, take $k\in\mathbb{N}$ with $\frac{1}{k}<\alpha$. We have
$$
\sum_{n=1}^{\infty}\left(\frac{1}{k}\right)^{a_n}=\sum_{j=1}^{\infty}j^j \left(\frac{1}{k}\right)^{j}
\geq\sum_{j=k}^{\infty}j^j \left(\frac{1}{k}\right)^{j}\geq \sum_{j=k}^{\infty}k^j \left(\frac{1}{k}\right)^{j}=\sum_{j=k}^{\infty} 1=\infty.
$$
Hence, $\sum_{n=1}^{\infty} \alpha^{a_n}=\infty$ for every $\alpha>0$. 
Consider the following:
\begin{enumerate}[resume]
    \item Let $(p_n)=(a_n)$. Then, $\ell_{p_n}\neq \ell_{\infty}$ since $\mathds{1}\not\in \ell_{p_n}$, but the inclusion $\ell_{p_n}\hookrightarrow\ell_{\infty}$ is not strictly singular by Theorem \ref{Thm.infinito2}.
    \item Let $(p_n)$ be bounded and $(q_n)$ with $q_n=p_n+\frac{1}{a_n}$. Then, $\ell_{p_n}\neq\ell_{q_n}$ by Proposition \ref{Nakano-igualdad}, but $\ell_{p_n}\hookrightarrow\ell_{q_n}$ is not strictly singular by Theorem \ref{Thm.base}.
\end{enumerate}

\end{exa}

We can study other classical properties for the inclusion $i:\ell_{p_n}\hookrightarrow \ell_{q_n}$, but most of them are easily proven.
Note that the inclusion $i:\ell_{p_n}\hookrightarrow \ell_{q_n}$ is not {\it compact}, neither \textit{$L$-weakly compact} (i.e. the image of the unit ball of $\ell_{p_n}$ is equi-integrable in $\ell_{q_n}$, cf. \cite{aliprantis_positive_2006,bennett_interpolation_1988,meyer-nieberg_banach_1991}) nor \textit{$M$-weakly compact} (i.e. $\lim_{n\rightarrow\infty} ||(x_n)||_{\ell_{q_n}} = 0$ for any $(x_n)$ norm bounded disjoint sequence in $\ell_{p_n}$), since the canonical unit sequence $(e_n)$ is normalized in every space $\ell_{p_n}$ and it has no convergent subsequence.
Also, $i$ is weakly compact if and only if $\ell_{q_n}$ is reflexive. Let us give a proof:

\begin{prop}
    Let the inclusion $i:\ell_{p_n}\hookrightarrow\ell_{q_n}$ hold. The inclusion $i$ is weakly compact if and only if $1<\varliminf_{n} q_n\leq\varlimsup_{n} q_n<\infty$.
\end{prop}

\begin{proof}
Indeed, if $\varlimsup_{n} q_n<\infty$, then $\ell_{q_n}$ is separable. Hence, Cor.3.7 in \cite{hernandez_remarks_2024} states that a subset $S\subset \ell_{q_n}$ is relatively weakly compact if and only if
$$
\lim_{\lambda\rightarrow 0}\sup_{(x_n)\in S} \frac{1}{\lambda} \sum_{n=1}^{\infty} \vert \lambda x_n\vert^{q_n}=0.
$$
So applying this to $S=B_{\ell_{p_n}}$ the unit ball of $\ell_{p_n}$ in $\ell_{q_n}$ and denoting $q_n^-:=\varliminf_{n} q_n>1$, we get
$$
\lim_{\lambda\rightarrow 0}\sup_{(x_n)\in B_{\ell_{p_n}}} \frac{1}{\lambda} \sum_{n=1}^{\infty} \vert \lambda x_n\vert^{q_n}
\leq\lim_{\lambda\rightarrow 0} \sup_{(x_n)\in B_{\ell_{p_n}}} \lambda^{q_n^--1} \sum_{n=1}^{\infty} \vert x_n\vert^{q_n}
\leq\lim_{\lambda\rightarrow 0} \lambda^{q_n^--1} \cdot C
=0
$$
since $B_{\ell_{p_n}}$ is bounded in $\ell_{q_n}$ by hypothesis.

On the other hand, if there exists a subsequence $(q_{n_k})$ with $q_{n_k}\rightarrow 1$ (respectively $q_{n_k}\rightarrow \infty$), for a further subsequence we have $q_{n_{k_j}}\leq 1+\frac{1}{j}$ (resp. $q_{n_{k_j}}\geq j$), then $\ell_{q_{n_{k_j}}}\simeq \ell_1$ by Proposition \ref{Nakano-igualdad} (resp. $\ell_{q_{n_{k_j}}}\simeq \ell_\infty$ by Proposition \ref{ele.infinito}) and, for every bounded inclusion $\ell_{p_n}\hookrightarrow \ell_{q_n}$, the bounded set $\{e_{n_{k_j}}\}_j$ in $\ell_{p_n}$ is not (relatively) weakly compact in $\ell_{q_{n_{k_j}}}=\ell_1$ by Schur's property (resp. is not weakly compact in $\ell_{q_{n_{k_j}}}=\ell_\infty$) and hence the inclusion $i$ is not weakly compact.
\end{proof}

This contrasts with the Nakano function spaces $\lpv$, where the behabior of these properties is much richer and varied (see \cite{hernandez_disjointly_2025, hernandez_remarks_2024}). 

\begin{center}
    ACKNOWLEDGEMENTS:

    I want to special thank Francisco L. Hernández and César Ruiz for their mentorship and help during and after my PhD, including the contents and presentation of this article.
\end{center}

\printbibliography

\end{document}